\documentclass[reqno]{amsart}
\usepackage{amsmath}
\usepackage{amsthm}
\usepackage{amssymb}
\usepackage{mathrsfs} 
\usepackage{bm}
\usepackage{hyperref}
\usepackage{comment}

\newcommand{\ZFC}{{\rm ZFC}}

\renewcommand{\emptyset}{\varnothing}

\newcommand{\R}{{\mathcal R}}
\renewcommand{\O}{{\mathcal O}}
\newcommand{\I}{{\mathcal I}}
\newcommand{\Part}{{\mathcal P}art}
\newcommand{\Ram}{{\mathcal R}am}
\newcommand{\pre}{{\rm pre}}

\newcommand{\restrict}{\upharpoonright}
\newcommand{\concat}{\mathbin{{}^\smallfrown}}

\newcommand{\<}{\langle}
\renewcommand{\>}{\rangle}

\newcommand{\st}{\mid}

\newcommand{\cf}{\mathop{\rm cf}}

\newcommand{\NS}{{\mathop{\rm NS}}}

\newcommand{\NSS}{{\mathop{\rm NSS}}}

\newcommand{\NSSub}{{\mathop{\rm NSSub}}}

\renewcommand{\and}{\mathop{\&}}

\newcommand{\strsub}{\prec}
\newcommand{\strsubeq}{\preceq}

\newtheorem{theorem}{Theorem}[section]
\newtheorem{lemma}[theorem]{Lemma}
\newtheorem{corollary}[theorem]{Corollary}
\newtheorem{proposition}[theorem]{Proposition}

\theoremstyle{definition}

\newtheorem{remark}[theorem]{Remark}

\newtheorem{definition}[theorem]{Definition}

\thanks{The authors would like to thank the anonymous referee for many helpful comments which greatly improved the article.}

\date{\today}

\begin{document}

\title{Two-cardinal ideal operators and indescribability}

\author[Brent Cody]{Brent Cody}
\address[Brent Cody]{ 
Virginia Commonwealth University,
Department of Mathematics and Applied Mathematics,
1015 Floyd Avenue, PO Box 842014, Richmond, Virginia 23284, United States
} 
\email[B. ~Cody]{bmcody@vcu.edu} 
\urladdr{http://www.people.vcu.edu/~bmcody/}

\author[Philip White]{Philip White}
\address[Philip White]{ 
George Washington University,
Department of Mathematics,
801 22nd St. NW, Room 739,
Washington, DC 20052
} 
\email[P. ~White]{pwhite30@gwu.edu} 

\begin{abstract}
A well-known version of Rowbottom's theorem for supercompactness ultrafilters leads naturally to notions of two-cardinal Ramseyness and corresponding normal ideals introduced herein. Generalizing results of Baumgartner, Feng and the first author, from the cardinal setting to the two-cardinal setting, we study hierarchies associated with a particular version of two-cardinal Ramseyness and a strong version of two-cardinal ineffability, as well as the relationships between these hierarchies and a natural notion of transfinite two-cardinal indescribability.
\end{abstract}


\subjclass[2010]{Primary 03E55; Secondary 03E05}

\keywords{}

\maketitle





\section{Introduction}\label{section_introduction}

One version of Ramsey's famous combinatorial theorem states that for every function $f:[\omega]^{2}\to 2$ there is an infinite set $H\subseteq\omega$ such that $H$ is \emph{homogeneous for $f$}, in the sense that $f\restrict[H]^2$ is constant. Since the work of Erd\H{o}s, Hajnal, Tarski, Rado and others \cite{MR95124, MR0065615, MR0008249, MR81864}, it has been well-known that certain generalizations of Ramsey's theorem to uncountable sets necessarily involve large cardinals. For example, we say that $\kappa>\omega$ is an \emph{ineffable cardinal} if for every function $f:[\kappa]^2\to\kappa$ with $f(a)<\min(a)$ for all $a\in[\kappa]^{2}$, there is an $H\subseteq\kappa$ that is stationary in $\kappa$ and homogeneous for $f$. Similarly, we say that $\kappa>\omega$ is a \emph{Ramsey cardinal} if for every function $f:[\kappa]^{<\omega}\to \kappa$ with $f(a)<\min(a)$ for all $a\in[\kappa]^{<\omega}$, there is a set $H\subseteq\kappa$ of size $\kappa$ that is homogeneous for $f$, that is $f\restrict[H]^n$ is constant for each $n<\omega$. The notions of ineffability and Ramseyness of cardinals leads naturally to the following definitions of the ineffability ideal operator $\I$ and the Ramsey ideal operator $\R$.

Suppose $\kappa$ is a regular cardinal and $I$ is an ideal on $\kappa$. We let $I^+=\{X\subseteq\kappa\st X\notin I\}$ be the corresponding collection of $I$-positive sets and $I^*=\{X\subseteq\kappa\st \kappa\setminus X\in I\}$ be the filter which is dual to $I$. We define new ideals $\I(I)$ and $\R(I)$ as follows. A set $X\subseteq\kappa$ is not in $\I(I)$ if and only if for every function $f:[X]^2\to\kappa$ with $f(a)<\min(a)$ for all $a\in[X]^2$, there is a set $H\in P(X)\cap I^+$ such that $f\restrict[H]^2$ is constant. Similarly, a set $X\subseteq\kappa$ is not in $\R(I)$ if and only if for every function $f:[X]^{<\omega}\to\kappa$ with $f(a)<\min(a)$ for all $a\in [X]^{<\omega}$, there is a set $H\in P(X)\cap I^+$ such that $f\restrict[H]^n$ is constant for all $n<\omega$. It follows from the work of Baumgartner \cite{MR0384553} that if $I\supseteq[\kappa]^{<\kappa}$ then $\I(I)$ is a normal ideal. The corresponding result for $\R(I)$ also holds, and is due to Feng \cite{MR1077260}.

By repeatedly applying the ideal operators $\I$ and $\R$ to various ideals, one is led naturally to consider the \emph{ineffability hierarchy} \cite{MR0540770} and the \emph{Ramsey hierarchy} \cite{MR1077260}. That is, if $\kappa$ is regular, $I$ is an ideal on $\kappa$ and $\O\in\{\I,\R\}$, we inductively define new ideals by letting
\begin{align*}
\O^0(I)&=I,\\
\O^{\alpha+1}(I)&=\O(\O^\alpha(I))\text{, and}\\
\O^\alpha(I)&=\bigcup_{\beta<\alpha}\O^\beta(I)\text{ when $\alpha$ is a limit.}
\end{align*}

We say that $\kappa$ is \emph{$\alpha$-ineffable} if and only if the ideal $\I^\alpha(\NS_\kappa)$ is nontrivial, that is, $\I^\alpha(\NS_\kappa)\neq P(\kappa)$.
Similarly, $\kappa$ is \emph{$\alpha$-Ramsey} if and only if the ideal $\R^\alpha([\kappa]^{<\kappa})$ is nontrivial.

The hierarchies of $\alpha$-ineffable and $\alpha$-Ramsey cardinals, and their relationship with various notions of indescribability \cite{MR0281606} have been extensively studied by Baumgartner \cite{MR0384553, MR0540770}, Feng \cite{MR1077260}, as well as the first author \cite{MR4206111} and the first author and Peter Holy \cite{MR4594301}. Although there is an extensive literature on two-cardinal combinatorial properties involving various notions of subtlety and ineffability \cite{MR1635559, MR2210149, MR2948437, MR2632006, MR854519, MR0325397, MR1864930, MR327518, MR2846027, MR3066742, MR3472180}, much less is known about two-cardinal analogues of Ramsey properties.

In this article, we introduce a well-behaved two-cardinal version of the Ramsey operator and generalize many results from the literature to our two-cardinal Ramsey operator as well as to a two-cardinal ineffable operator previously studied by Kamo \cite{MR1419442} as well as Abe and Usuba \cite{MR2948437}.

\section{Two-cardinal ideal operators associated to ineffability and partition properties}

\subsection{Stationarity, strong stationarity and strong normality}



Suppose $\kappa$ is regular and $A$ is a set of ordinals with $\kappa\leq|A|$. We write $P_\kappa A$ or $[A]^{<\kappa}$ to denote the collection of subsets of $A$ of cardinality less than $\kappa$. A set $S\subseteq P_\kappa A$ is \emph{unbounded in $P_\kappa A$} if for every $x\in P_\kappa A$ there is a $y\in S$ with $x\subseteq y$.  It is easy to see that the collection 
\[I_{\kappa, A}=\{X\subseteq P_\kappa A\st\text{$X$ is not unbounded}\}\]
is a nontrivial ideal on $P_\kappa A$. Moreover, $I_{\kappa,A}^+$ is the set of unbounded subsets of $P_\kappa A$ and the filter generated by the collection $\{\hat{x}\st x\in P_\kappa A\}$, where $\hat{x}=\{y\in P_\kappa A: x\subseteq y\}$, equals the filter $I_{\kappa,A}^*$ dual to $I_{\kappa,A}$. Also notice that because $\kappa$ is assumed to be regular, for any $\gamma<\kappa$ and any sequence $\<X_\alpha\st\alpha<\gamma\>$ with $A_\alpha\in I_{\kappa,A}$ for $\alpha<\gamma$, we have $\bigcup_{\alpha<\gamma}A_\alpha\in I_{\kappa,A}$.

Jech defined two-cardinal notions of closed unboundedness and stationarity as follows. A set $C\subseteq P_\kappa A$ is \emph{closed} if whenever $\{c_\alpha\st \alpha<\gamma\}$ is a $\subseteq$-increasing chain in $C$ of length less than $\kappa$, it follows that $\bigcup_{\alpha<\gamma} c_\alpha\in C$. A set $C\subseteq P_\kappa A$ is \emph{club in $P_\kappa A$} if it is closed and unbounded in $P_\kappa A$, and a set $S\subseteq P_\kappa A$ is \emph{stationary in $P_\kappa A$} if $S\cap C\neq\emptyset$ for all clubs $C$ in $P_\kappa A$. Jech showed that when $\kappa$ is regular the set
\[\NS_{\kappa,A}=\{X\subseteq P_\kappa A\st \text{$X$ is nonstationary}\}\]
is a nontrivial normal ideal on $P_\kappa A$, meaning that for every $S\in \NS_{\kappa,A}^+$ and every function $f:S\to \lambda$, with $f(x)\in x$ for all $x\in S$, there is a $T\subseteq S$ which is stationary in $P_\kappa A$ such that $f\restrict T$ is constant. It is easy to see that $\NS_{\kappa,A}^*$ is the filter generated by the club subsets of $P_\kappa A$.

The ideals $I_{\kappa,A}$ and $\NS_{\kappa,A}$ are defined using the ordering $(P_\kappa A,\subseteq)$. When $\kappa$ is inaccessible it is often advantageous to work with a different ordering. If $x\in P_\kappa A$ we let $\kappa_x=|x\cap \kappa|$. For $x,y\in P_\kappa A$ we define $x\strsub y$ if and only if $x\subseteq y$ and $|x|<|y\cap \kappa|$ (equivalently $x\in P_{\kappa_y}y$). See \cite{MR1994835} for an introductory discussion of $\strsub$ and the notion of \emph{strong normality}, which we also define below. Notice that if $\kappa$ is inaccessible, a set $S\subseteq P_\kappa A$ is unbounded if and only if for every $x\in P_\kappa A$ there is a $y\in S$ with $x\strsub y$. Since we will focus on the case in which $\kappa$ is inaccessible, we don't loose anything by working with $I_{\kappa,\lambda}$ rather than its $\strsub$ counterpart. 

 We say that a set $C\subseteq P_\kappa A$ is a \emph{weak club in $P_\kappa A$} if $C$ is $\strsub$-unbounded in $P_\kappa A$ and whenever $C\cap P_{\kappa_x}x\in I_{\kappa_x,x}^+$, for some $x\in P_\kappa A$, we have $x\in C$. It is straightforward to see that when $C\subseteq P_\kappa A$ is a weak club in $P_\kappa A$ there is a function $f:P_\kappa A\to P_\kappa A$ such that the set
\[C_f:=\{x\in P_\kappa A\st x\cap\kappa\neq\emptyset\land f"P_{\kappa_x}x\subseteq P_{\kappa_x}x\}\]
is contained in $C$. Furthermore, for any such function $f$, the set $C_f$ is a weak club subset of $P_\kappa A$.

For $S\subseteq P_\kappa A$, a function $f:S\to P_\kappa A$ is said to be \emph{${\strsub}$-regressive on $S$} if $f(x)\strsub x$ for all $x\in S$. An ideal $I$ on $P_\kappa A$ is \emph{strongly normal} if for all $S\in I^+$ and all ${\strsub}$-regressive functions $f:S\to P_\kappa A$ there is a $T\in P(S)\cap I^+$ such that $f$ is constant on $T$. It is easy to see that an ideal $I$ on $P_\kappa A$ is strongly normal if and only if for every sequence $\vec{A}=\<A_x\st x\in P_\kappa A\>$ with $A_x\in I^*$ for all $x\in P_\kappa A$, the \emph{$\strsub$}-diagonal intersection 
\[\bigtriangleup_\strsub\{A_x\st x\in P_\kappa A\}=\{y\in P_\kappa A\st y\in \bigcap_{x\strsub y}A_x\}\]
is in $I^*$. A set $S\subseteq P_\kappa A$ is \emph{strongly stationary} if $S\cap C\neq\emptyset$ for all weak clubs $C\subseteq P_\kappa A$. Let us note that this definition of strongly stationary differs slightly from that used elsewhere in the literature, however, it is equivalent to the corresponding notions used in \cite{MR1074449,MR4082998} when $\kappa$ is a Mahlo cardinal (see \cite[Fact 2.1]{CLHZ2024}). The \emph{non--strongly stationary ideal on $P_\kappa A$} is the collection.
\[\NSS_{\kappa, A}=\{X\subseteq P_\kappa A\st\text{$X$ is not strongly stationary}\}.\]
When $\kappa$ is Mahlo, it follows that $\NS_{\kappa, A}\subsetneq\NSS_{\kappa, A}$ and that $\NSS_{\kappa, A}$ is the minimal strongly normal ideal (see \cite[Section 6]{MR1074449} or \cite[Corollary 2.3]{MR4082998}). See \cite[Section 3]{MR1297401} for information on the relationship between $\NSS_{\kappa,\lambda}$ and $\NS_{\kappa,\theta}$ when $\theta=\lambda^{<\kappa}$.

\subsection{Two-cardinal ideal operators associated to ineffability and partition properties}

Kamo \cite{MR1419442} studied several ideal operators associated to notions of two-cardinal ineffability and parition properties introduced by Jech \cite{MR0325397}. While the results of Jech \cite{MR0325397}, Menas \cite{MR2940301}, Magidor \cite{MR327518} and others focus on ineffability and partition properties defined using the ordering $(P_\kappa A,\subseteq)$, Kamo introduced similar notions defined using $(P_\kappa A,\strsub)$. Let us review the relevant definitions.

Given a set $S\subseteq P_\kappa A$, a sequence $\vec{S}=\<S_x\st x\in S\>$ is called an \emph{$(S,\subset)$-list} if $S_x\subseteq x$ for all $x\in S$, and $\vec{S}$ is called an \emph{$(S,\strsub)$-list} if $S_x\subseteq P_{\kappa_x} x$ for all $x\in S$. If $\vec{S}$ is an $(S,\subset)$-list, we say that $H\subseteq S$ is \emph{homogeneous for $\vec{S}$} if whenever $x,y\in H\cap S$ and $x\subsetneq y$ we have $S_x=S_y\cap x$. If $\vec{S}$ is an $(S,\strsub)$-list, we say that $H\subseteq S$ is \emph{homogeneous for $\vec{S}$} if whenever $x,y\in H\cap S$ and $x\strsub y$ we have $S_x=S_y\cap P_{\kappa_x}x$. In the following definitions we will consider $2$-colorings of sets of the form
\[[S]^2_\lhd=\{(x,y)\st x,y\in S\ \land\ x\lhd y\}\]
where $\lhd$ is some ordering on $P_\kappa A$. Here we want to consider colorings of sets of $\lhd$-increasing \emph{ordered} pairs because, in the cases we are interested in, namely $\lhd\in\{\subsetneq,\strsub\}$, it is often the case that large homogeneous sets for colorings of unordered pairs do not exist; see the paragraph after Definition 4.5 in \cite{MR0325397} for details. Given a function $f:[S]^2_\lhd\to 2$ we say that $H\subseteq S$ is \emph{homogeneous for $f$} if $f\restrict [H]^2_\lhd$ is a constant function.

\begin{definition}
Suppose $I$ is an ideal on $P_\kappa A$ and $\lhd$ is some ordering on $P_\kappa A$. We define new ideals $\I_\subset(I)$, $\I_\strsub(I)$ and $\Part_\lhd(I)$ on $P_\kappa A$ as follows.
\begin{enumerate}
\item $\I_\subset(I)$ is the ideal on $P_\kappa A$ such that $S\in \I_\subset(I)^+$ if and only if every $(S,\subset)$-list has a homogeneous set $H\subseteq S$ in $I^+$.
\item $\I_\strsub(I)$ is the ideal on $P_\kappa A$ such that $S\in\I_\strsub(I)^+$ if and only if every $(S,\strsub)$-list has a homogeneous set $H\subseteq S$ in $I^+$.
\item $\Part_\lhd(I)$ is the ideal on $P_\kappa A$ such that $S\in \Part_\lhd(I)^+$ if and only if every function $f:[S]^2_\lhd\to 2$ has a homogeneous set $H\subseteq S$ in $I^+$.
\end{enumerate}
\end{definition}

It is not too difficult to see that when $I$ is a normal ideal on $P_\kappa A$, it follows that $\I_\subset(I)$ and $\Part_\subset(I)$ are normal ideals and $\I_\strsub(I)$ and $\Part_\strsub(I)$ are strongly normal. Furthermore, if $I\supseteq I_{\kappa,A}$ is any ideal on $P_\kappa A$, then $\I_\strsub(I)$ is strongly normal (see the paragraph after Definition 3.1 in \cite{MR1419442}). For more details and related results see \cite[Section 3]{MR1419442}.

 Note that, in Jech's terminology, $\kappa$ is $\lambda$-ineffable if and only if $\I_\subset(\NS_{\kappa,\lambda})$ is a nontrivial ideal. Similarly, $S\subseteq P_\kappa A$ is an \emph{ineffable subset of $P_\kappa A$} if and only if $S\in \I(\NS_{\kappa,A})^+$. We say that $S$ has the \emph{$\lhd$-partition property} if $S\in\Part_\lhd(\NS_{\kappa,A})^+$. Johnsson \cite{MR1111312} showed that when $\cf(\lambda)\geq\kappa$, a set $S\subseteq P_\kappa A$ is ineffable if and only if it has the $\strsub$-partition property; hence $\I_\subset(\NS_{\kappa,A})=\Part_{\strsub}(\NS_{\kappa,A})$ (see \cite[Fact 1.13]{MR2948437}). The relationships between the operators $\I_\subset$, $\I_\strsub$, $\Part_\subsetneq$ and $\Part_\strsub$ have been further explored by Kamo \cite{MR1419442} as well as Abe and Usuba \cite{MR2948437}.

\subsection{A two-cardinal ideal operators associated to Ramseyness}

Suppose $\kappa$ is regular and $A$ is a set of ordinals with $\kappa\leq|A|$. Suppose $S\subseteq P_\kappa A$. Given a tuple $\vec{x}=(x_1,\ldots,x_n)\in S^n$, with $x_1\strsub\cdots\strsub x_n$, we will identify $\vec{x}$ with the $\strsub$-increasing enumeration of its entries. Given $S\subseteq P_\kappa A$ and $n<\omega$, we let
\[[S]^n_\strsub=\{(x_1,\ldots,x_n)\in S^n\st x_1\strsub\cdots\strsub x_n\}\]
and
\[[S]^{<\omega}_\strsub=\bigcup_{n<\omega}[S]^n_{\strsub}.\]

A function $f:[S]^{<\omega}_\strsub\to P_\kappa A$ is called \emph{${\strsub}$-regressive} if $f(x_1,\ldots,x_n)\strsub x_1$ for all $(x_1,\ldots,x_n)\in [S]^{<\omega}_\strsub$.

The following is a straightforward generalization of a standard fact about supercompactness ultrafilters which is used to prove that supercompact Prikry forcing satisfies the Prikry property (see \cite[Section 1.4]{MR2768695}).

\begin{proposition}
Suppose $U$ is a $\kappa$-complete normal fine ultrafilter on $P_\kappa\lambda$ and $f:[P_\kappa\lambda]^{<\omega}_\strsub\to P_\kappa\lambda$ is a $\strsub$-regressive function. Then there is a set $H\in U$ which is homogeneous for $f$, meaning that $f\restrict[H]^n_\strsub$ is constant for all $n<\omega$.
\end{proposition}

\begin{proof}
It suffices to show that for each $n\in\omega\setminus\{0\}$ there is an $H_n\in U$ such that $f\restrict[H_n]^n_\strsub$ is constant, because then $H=\bigcap_{n<\omega}H_n\in U$ will be the desired homogeneous set for $f$.

We will prove by induction on $n$ that for every $\strsub$-regressive $F:[P_\kappa\lambda]^{n}_\strsub\to P_\kappa\lambda$ there is $H\in U$ such that $F\restrict[H]^n_{\strsub}$ is constant. This holds for $n=1$ by the strong normality of $U$. Suppose its true for $n$. Let $F:[P_\kappa\lambda]^{n+1}_\strsub\to P_\kappa\lambda$ be $\strsub$-regressive. For each $x\in P_\kappa\lambda$ define
$F_x:[P_\kappa\lambda]^n_{\strsub}\to P_\kappa\lambda$ by
\begin{align*}
    F_x(x_1,\ldots,x_n)=\begin{cases}
        F(x,x_1,\ldots,x_n) & \text{if $x\strsub x_1$}\\
        0 & \text{o.w.}
    \end{cases}
\end{align*}
Let $C_x=\{y\in P_\kappa\lambda\st x\strsub y\}$. Then $C_x$ is club and for all $(x_1,\ldots,x_n)\in[C_x]^n$ we have $F_x(x_1,\ldots,x_n)=F(x,x_1,\ldots,x_n)$. By our inductive hypothesis, it follows that for each $x\in P_\kappa\lambda$ there is $H_x'\in U$ such that $F_x\restrict[H_x']^n_\strsub$ is constant. Furthermore, $H_x=H_x'\cap C_x\in U$, $F_x\restrict[H_x]^n_\strsub$ is constant. For $(x_1,\ldots,x_n)\in [H_x]^n_\strsub$ we let $i_x=F_x(x_1,\ldots,x_n)=F(x,x_1,\ldots,x_n)$ denote this constant value and note that $i_x\strsub x$. Now let 
\begin{align*}
    H'&=\bigtriangleup_\strsub\{H_x\st x\in P_\kappa\lambda\}\\
        &=\{y\in P_\kappa\lambda\st y\in\bigcap_{z\strsub y}H_z\}.
\end{align*}
We have $H'\in U$ by the strong normality of $U$. If $(x,x_1,\ldots,x_n)\in[H']^{n+1}_\strsub$ then $(x_1,\ldots,x_n)\in[H_x]^n$ and thus $F(x,x_1,\ldots,x_n)=F_x(x_1,\ldots,x_n)=i_x$. Since $x\mapsto i_x$ is regressive on $H'\in U$ there is $H\subseteq H'$ in $U$ such that for all $x\in H$ we have $i_x=i$, where $i$ is some fixed element of $P_\kappa\lambda$. Thus, if $(x_0,\ldots,x_n)\in [H]^{n+1}_\strsub$ then $F(x_0,\ldots,x_n)=i_{x_0}=i$.
\end{proof}

The previous result motivates the following definitions, which resemble characterizations of the one-cardinal Ramsey operator studied in \cite{MR0540770, MR4206111, MR4594301, MR1077260, MR2817562}.

\begin{definition}\label{definition_two_cardinal_ramsey_operator}
Suppose $\kappa$ is regular, $A$ is a set of ordinals with $\kappa\leq|A|$ and $I\supseteq I_{\kappa,A}$ is an ideal on $P_\kappa A$. We define ideals $\Ram(I)$ and $\Ram_\strsub(I)$ on $P_\kappa A$ as follows.
\begin{enumerate}
\item $\Ram(I)$ is the ideal on $P_\kappa A$ such that $S\in\Ram(I)^+$ if and only if every function $f:[S]^{<\omega}_\strsub\to 2$ has a homogeneous set $H\subseteq S$ in $I^+$, meaning that $f\restrict[H]^n$ is constant for all $n<\omega$.
\item $\Ram_\strsub(I)$ is the ideal on $P_\kappa A$ such that $S\in \Ram_\strsub(I)^+$ if and only if every $\strsub$-regressive function $f:[S]^{<\omega}_\strsub\to P_\kappa A$ has a homogeneous set $H\subseteq S$ in $I^+$, meaning that $f\restrict[H]^n$ is constant for all $n<\omega$.
\end{enumerate}
\end{definition}

It is easy to verify that $\Ram(I)\subseteq\Ram_\strsub(I)$, but unfortunately not much else is known about $\Ram(I)$. For example, generalizing from the one-cardinal case, one would like to show that Definition \ref{definition_two_cardinal_ramsey_operator}(1) and Definition \ref{definition_two_cardinal_ramsey_operator}(2) are equivalent when $I\supseteq \NS_{\kappa,A}$ (see \cite{MR4206111}); however, it remains open whether this can be done. In what follows we will focus on $\Ram_\strsub(I)$ rather than $\Ram(I)$.

Let us prove that $\Ram_\strsub(I)$ is a strongly normal ideal on $P_\kappa A$ whenever $I$ is an ideal on $P_\kappa A$.

\begin{theorem}
Suppose $\kappa$ is regular, $A$ is a set of ordinals with $\kappa\leq |A|$ and $I$ is an ideal on $P_\kappa A$. Then $\Ram_\strsub(I)$ is a strongly normal ideal on $P_\kappa A$.
\end{theorem}

\begin{proof}
We follow \cite[Theorem 2.1]{MR1077260}. Without loss of generality, suppose $P_\kappa A\notin\Ram_\sqsubset(I)$, so that $\Ram_\sqsubset(I)$ is nontrivial. Suppose $X\in\Ram_\sqsubset(I)^+$ and $h:X\to P_\kappa A$ is ${\sqsubset}$-regressive. For the sake of contradiction, suppose that for all $y\in P_\kappa A$ we have $h^{-1}(\{y\})\in \Ram_\sqsubset(I)$. For each $y\in P_\kappa A$, fix a ${\sqsubset}$-regressive $f_y:[h^{-1}(\{y\})]^{<\omega}_\sqsubset\to P_\kappa A$ and a weak club $C_y$ which witness $h^{-1}(\{y\})\in\Ram_\sqsubset(I)$. Let $\pi:P_\kappa A\times P_\kappa A\to P_\kappa A$ be a bijection and notice that $D=\{x\in P_\kappa A\st \pi"P_{\kappa_x}x\times P_{\kappa_x}x\subseteq P_{\kappa_x}x\}$ is a weak club. Then $C=\bigtriangleup_\sqsubset\{C_y\st y\in P_\kappa A\}\cap D$ is a weak club in $P_\kappa A$, and hence $X\cap C\in\Ram_\sqsubset(I)^+$. Define a ${\sqsubset}$-regressive function $f:[X\cap C]^{<\omega}_\sqsubset\to P_\kappa A$ by letting
\begin{align*}
    f(\{x\})&=\pi(h(x),f_{h(x)}(\{x\}))\\
    f(x_1,\ldots,x_n)&=\begin{cases}
        f_{h(x_1)}(x_1,\ldots,x_n) & \text{if $h(x_1)=\cdots=h(x_n)$}\\
        0 & \text{otherwise}
    \end{cases}
\end{align*}
Since $X\in \Ram_\sqsubset(I)^+$, there is an $H\in P(X\cap C)\cap I^+$ which is homogeneous for $f$. Let $z\in P_\kappa A$ be such that $f(\{x\})=z$ for all $x\in H$. Thus, there is some $y\in P_\kappa A$ such that $h(x)=y\sqsubset x$ for all $x\in H$, and by definition of diagonal intersection, we have $H\cap \{a\in P_\kappa A\st y\sqsubset a\}\subseteq C_y\cap h^{-1}(\{y\})$. By definition of $f_y$, it follows that $H$ is not homogeneous for $f_y$. Since $H$ is homogeneous for $f$ but not homogeneous for $f_y$, and since $f\restrict([H]^{<\omega}\setminus [H]^1)=f_y\restrict([H]^{<\omega}\setminus [H]^1)$, it follows that there are $x_1,x_2\in H$ such that $f_y(\{x_1\})\neq f_y(\{x_2\})$, but this is not possible because it implies $\pi(y,f_y(\{x_1\}))\neq\pi(y,f_y(\{x_2\}))$ and hence $f(\{x_1\})\neq f(\{x_2\})$, a contradiction.
\end{proof}

Since the non-strongly stationary ideal $\NSS_{\kappa,A}$ is the minimal strongly normal ideal on $P_\kappa A$ \cite{MR1074449}, we easily obtain the following corollary.

\begin{corollary}\label{proposition_weak_club_for_free}
Suppose $\kappa$ is regular, $A$ is a set of ordinals with $\kappa\leq |A|$ and $I\supseteq I_{\kappa, A}$ is an ideal on $P_\kappa A$. Then $S\in \Ram_\strsub(I)^+$ if and only if $S\cap C\in \Ram_\strsub(A)^+$ for all weak clubs $C$ in $P_\kappa A$.
\end{corollary}

Feng \cite[Theorem 2.3]{MR1077260} gave a characterization of the one-cardinal Ramsey operator in terms of \emph{$(\omega,S)$}-sequences. We would like to generalize this characterization to the two-cardinal operator $\Ram_\strsub$. Given $S\subseteq P_\kappa A$, an \emph{$(\omega,S,\strsub)$-list} is a function $\vec{S}:[S]^{<\omega}_\strsub\to P(P_\kappa A)$ such that $\vec{S}(x_1,\ldots,x_n)\subseteq P_{\kappa_{x_1}}x_1$ for all $(x_1,\ldots,x_n)\in[S]^{<\omega}_\strsub$. We say that a set $H\subseteq S$ is \emph{homogeneous} for $\vec{S}$ if for all $n<\omega$ and all $(x_1,\ldots,x_n)$ and $(y_1,\ldots,y_n)$ in $[H]^n_{\strsub}$ with $x_1\strsubeq y_1$ we have $S(y_1,\ldots,y_n)\cap P_{\kappa_{x_1}}x_1=S(x_1,\ldots,x_n)$.


\begin{proposition}
Suppose $\kappa$ is a cardinal, $A$ is a set of ordinals with $\kappa\leq|A|$, $I\supseteq I_{\kappa, A}$ is an ideal on $P_\kappa A$ and $S\subseteq P_\kappa A$. The following are equivalent.
\begin{enumerate}
    \item $S\in \Ram_\strsub(I)^+$
    \item For every $(\omega,S,\strsub)$-list $\vec{S}:[S]^{<\omega}_\strsub\to P(P_\kappa A)$ there is an $H\in P(S)\cap I^+$ which is homogeneous for $\vec{S}$.
    \item For any $(\omega,S,\strsub)$-list $\vec{S}:[S]^{<\omega}_\strsub\to P(P_\kappa A)$ there is an $H\in P(S)\cap I^+$ and a sequence $\<S_n\st 1<n<\omega\>$ of subsets of $P_\kappa A$ such that for all $n$, for all $(x_1,\ldots,x_n)\in[H]^n_\strsub$, we have $\vec{S}(x_1,\ldots,x_n)=S_n\cap P_{\kappa_{x_1}}x_1$.
\end{enumerate}
\end{proposition}

\begin{proof}
It is easy to see that (2) and (3) are equivalent. Let us show that (1) and (2) are equivalent. 

For (1) implies (2), suppose $\vec{S}$ is an $(\omega,S,\strsub)$-list. Since 
\[C=\{x\in P_\kappa A\st \text{$x\cap\kappa$ is a limit ordinal}\}\]
is club in $P_\kappa A$, we can assume that $S\subseteq C$. Define $g:[S]^{<\omega}_\strsub\to P_\kappa A$ such that for all $(x_1,\ldots,x_{2n})\in [S]^{2n}_\strsub$, setting $a=(x_1,\ldots,x_n)$ and $b=(x_{n+1},\ldots,x_{2n})$ we have $g(x_1,\ldots,x_{2n})=\emptyset$ if $\vec{S}(a)=\vec{S}(b)\cap P_{\kappa_{x_1}}x_1$, and $g(x_1,\ldots,x_{2n})=z\cup\{\kappa_z\}$ where $z$ is some element of $\vec{S}(a)\bigtriangleup\vec{S}(b)\cap P_{\kappa_{x_1}}x_1$ if $\vec{S}(a)\neq\vec{S}(b)\cap P_{\kappa_{x_1}}x_1$.\footnote{We use $z\cup\{\kappa_z\}$ in the second case so that the value of $g(x_1,\ldots,x_{2n})$, where $(x_1,\ldots,x_{2n})\in[S]^{2n}_\strsub$, can be used to determine which case $(x_1,\ldots,x_{2n})$ fall into.} Notice that since $x_1\in S\subseteq C$, it follows that, in the second case above $z\cup\{\kappa_z\}\in P_{\kappa_{x_1}}x_1$, and thus $g$ is $\strsub$-regressive. Let $H\in P(S)\cap I^+$ be homogeneous for $g$.

Let us show that if $a,b\in[H]^n_\strsub$ are such that $a_n\strsub b_1$ then $\vec{S}(a)=\vec{S}(b)\cap P_{\kappa_{x_1}}x_1$. Suppose $\vec{S}(a)\neq\vec{S}(b)\cap P_{\kappa_{x_1}}x_1$. Then $g(a\concat b)=z\cup\{\kappa_z\}$ where $z\in\vec{S}(a)\bigtriangleup\vec{S}(b)\cap P_{\kappa_{x_1}}x_1$. Without loss of generality, say $z\in\vec{S}(a)\setminus\vec{S}(b)$. Let $c\in[H]^n_\strsub$ be such that $b_n\strsub c_1$. Then, by the homogeneity of $H$, we have $g(b\concat c)=z\cup\{\kappa_z\}$, and thus by definition of $g$ we have $z\in \vec{S}(b)\bigtriangleup\vec{S}(c)\cap P_{\kappa_{x_1}}x_1$. Since $z\notin\vec{S}(b)$ we have $z\in\vec{S}(a)\cap\vec{S}(c)$, which implies $z\notin\vec{S}(a)\bigtriangleup\vec{S}(c)\cap P_{\kappa_{x_1}}x_1$. However, by homogeneity of $H$, it follows that $g(a\concat c)=z\cup\{\kappa_z\}\neq \emptyset$, and thus $z\in\vec{S}(a)\bigtriangleup\vec{S}(b)\cap P_{\kappa_{x_1}}x_1$, a contradiction.

Now, let $a,b\in [H]^n_{\strsub}$ be such that $a_1\strsubeq b_1$. Choose $c\in[H]^n_\strsub$ with $a_n,b_n\strsub c_1$. Then $\vec{S}(a)=\vec{S}(c)\cap _{\kappa_{a_1}}a_1$ and $\vec{S}(b)=\vec{S}(c)\cap P_{\kappa_{b_1}}b_1$, which implies 
\begin{align*}
\vec{S}(b)\cap P_{\kappa_{a_1}}a_1&=(\vec{S}(c)\cap P_{\kappa_{b_1}}b_1)\cap P_{\kappa_{a_1}}a_1\\
	&=\vec{S}(c)\cap P_{\kappa_{a_1}}a_1\\
	&=\vec{S}(a).
\end{align*}

For (2) implies (1), let $f:[S]^{<\omega}_\strsub\to P_\kappa A$ be a ${\strsub}$-regressive function and let $C\subseteq P_\kappa A$ be a weak club. We define an $(\omega,S\cap C,\strsub)$-sequence $\vec{S}$ as follows. For each $a\in [S\cap C]^{<\omega}_\strsub$ let $\vec{S}(a)=\{f(a)\}\subseteq P_{\kappa_{a_1}}a_1$. Let $H\in P(S\cap C)\cap I^+$ be homogeneous for $\vec{S}$. Then $H$ is also homogeneous for $f$.
\end{proof}


Next we demonstrate that the nontriviality of the ideal $\Ram_\strsub(I)$ naturally leads to the existence of nonlinear sets of indiscernibles for certain structures in countable languages.

\begin{definition}
If $\mathcal{M}$ is a structure in a countable language and $P_\kappa A\subseteq \mathcal{M}$ we say that $H\subseteq P_\kappa A$ is a \emph{set of indiscernibles for $\mathcal{M}$} if for every $n<\omega$ and for all $(x_1,\ldots, x_n),(y_1,\ldots,y_n)\in[H]^n_{\strsub}$ we have
\[\mathcal{M}\models\varphi(x_1,\ldots,x_n)\text{ if and only if }\mathcal{M}\models\varphi(y_1,\ldots,y_n)\]
for all first-order $\varphi$ in the language of $\mathcal{M}$ with exactly $n$ free variables.
\end{definition}

\begin{proposition}\label{proposition_indiscernibles_from_2_coloring}
Suppose $\kappa$ is a cardinal, $A$ is a set of ordinals with $\kappa\leq|A|$ and $I\supseteq I_{\kappa,A}$ is an ideal on $P_\kappa A$. If every function $f:[P_\kappa A]^{<\omega}_\strsub\to 2$ has a homogeneous set in $I^+$ then every structure $\mathcal{M}$ in a countable language with $P_\kappa A\subseteq \mathcal{M}$ has a set of indiscernibles $H\in I^+$.
\end{proposition}

\begin{proof}
Using an argument similar to that of \cite[Proposition 7.14(c)]{MR1994835}, it is easy to see that our assumption implies that for every $\gamma<\kappa$ every function $f:[P_\kappa A]^{<\omega}_\strsub\to\gamma$ has a homogeneous set $H\in I^+$. Let $\mathcal{M}$ be a structure in a countable language with $P_\kappa A\subseteq\mathcal{M}$ and let $\Phi_n$ denote the collection of all first order formulas in the language of $\mathcal{M}$ with exactly $n$ free variables. Define a function $f$ with domain $[P_\kappa A]^{<\omega}_\strsub$ by letting
\[f(x_1,\ldots,x_n)=\{\varphi\in \Phi_n\st \mathcal{M}\models\varphi(x_1,\ldots,x_n)\}.\]
Since $|P(\Phi_n)|<\kappa$, $f$ has a homogeneous set $H\in I^+$. It is easy to verify that $H$ is a set of indiscernibles for $M$.
\end{proof}

\section{Generalizing results of Baumgartner and Feng}

The ineffability hierarchy and the Ramsey hierarchy, which were introduced by Baumgartner \cite{MR0540770} and Feng \cite{MR1077260} respectively, can be obtained by iterating the associated ideal operators. In the present section we consider two-cardinal versions of these hierarchies and investigate their relationship with a notion of transfinite two-cardinal indescribability which generalizes previously studied notions \cite{MR2817562, MR3894041, MR4082998}.

\subsection{Transfinite two-cardinal indescribability}

Let us now generalize a notion of transfinite indescribability introduced in \cite{MR3894041}, and further utilized in \cite{MR4094556, MR4693983}\footnote{Let us note that another notion of transfinite indescribability defined in terms of games was introduced by Welch and Sharpe \cite{MR2817562}.} to the two-cardinal context.

For the reader's convenience, let us discuss the notion of $\Pi^1_\xi$ formula introduced in \cite{MR3894041}. Recall that a formula of second-order logic is $\Pi^1_0$, or equivalently $\Sigma^1_0$, if it does not have any second-order quantifiers, but it may have finitely-many first-order quantifiers and finitely-many first and second-order free variables. Bagaria inductively defined the notion of $\Pi^1_\xi$ formula for any ordinal $\xi$ as follows. A formula is $\Sigma^1_{\xi+1}$ if it is of the form
\[\exists X_0\cdots\exists X_k\varphi(X_0,\ldots,X_k)\]
where $\varphi$ is $\Pi^1_\xi$, and a formula is $\Pi^1_{\xi+1}$ if it is of the form 
\[\forall X_0\cdots\forall X_k\varphi(X_0,\ldots, X_k)\]
where $\varphi$ is $\Sigma^1_\xi$. If $\xi$ is a limit ordinal, we say that a formula is $\Pi^1_\xi$ if it is of the form 
\[\bigwedge_{\zeta<\xi}\varphi_\zeta\]
where $\varphi_\zeta$ is $\Pi^1_\zeta$ for all $\zeta<\xi$ and the infinite conjunction has only finitely-many free second-order variables. We say that a formula is $\Sigma^1_\xi$ if it is of the form
\[\bigvee_{\zeta<\xi}\varphi_\zeta\]
where $\varphi_\zeta$ is $\Sigma^1_\zeta$ for all $\zeta<\xi$ and the infinite disjunction has only finitely-many free second-order variables.

Suppose $\kappa$ is a cardinal and $A$ is a set of ordinals with $\kappa\leq|A|$. We define a two-cardinal version of the cumulative hierarchy up to $\kappa$ as follows:
\begin{align*}
V_0(\kappa,A)&=A,\\
V_{\alpha+1}(\kappa,A)&=P_\kappa(V_\alpha(\kappa,A))\cup V_\alpha(\kappa,A)\text{ and}\\
V_\alpha(\kappa,A)&=\bigcup_{\beta<\alpha} V_\beta(\kappa,A)\text{ for $\alpha$ a limit.}
\end{align*}

Clearly $V_\kappa\subseteq V_\kappa(\kappa,A)$ and if $A$ is transitive then so is $V_\alpha(\kappa,A)$ for all $\alpha\leq\kappa$. See \cite[Section 4]{MR808767} for a discussion of the restricted axioms of $\ZFC$ satisfied by $V_\kappa(\kappa,\lambda)$ when $\kappa$ is inaccessible.

\begin{definition}\label{definition_indescribability}
Suppose $\kappa$ is a regular cardinal and $A$ is a set of ordinals with $\kappa\leq|A|$. We say that a set $S\subseteq P_\kappa A$ is $\Pi^1_\xi$-indescribable in $P_\kappa A$ if whenever $(V_\kappa(\kappa,A),\in,R)\models\varphi$ where $k<\omega$, $R\subseteq V_\kappa(\kappa,A)$ and $\varphi$ is a $\Pi^1_\xi$ sentence, there is an $x\in S$ such that
\[x\cap\kappa=|x\cap\kappa|\text{ and }(V_{\kappa_x}(\kappa_x,x),\in,R\cap V_{\kappa_x}(\kappa_x,x))\models\varphi.\]
We define the $\Pi^1_\xi$-indescribability ideal on $P_\kappa A$ to be the collection \[\Pi^1_\xi(\kappa,A)=\{X\subseteq P_\kappa A\st \text{$X$ is not $\Pi^1_\xi$-indescribable in $P_\kappa A$}\}.\] 
\end{definition}

Let us note that the first author proved \cite{MR4082998} that $\Pi^1_0(\kappa,A)=\NSS_{\kappa,A}$. For notational convenience we let $\Pi^1_{-1}(\kappa,A)=\I_{\kappa,A}$.

Abe proved \cite[Lemma 4.1]{MR1635559} that $\Pi^1_n(\kappa,A)$ is a strongly normal ideal on $P_\kappa A$ for $n<\omega$ (see \cite{MR4082998} for some additional characterizations of $\Pi^1_n(\kappa,A)$). A straightforward generalization of the argument for \cite[Proposition 4.4]{MR3894041} establishes the following.

\begin{proposition}
Suppose $\kappa$ is a regular cardinal and $A$ is a set of ordinals with $\kappa\leq|A|$. Then $\Pi^1_\xi(\kappa,A)$ is a strongly normal ideal on $P_\kappa A$.
\end{proposition}

\subsection{Iterating two-cardinal ideal operators}

Given an ideal $I$ on $P_\kappa A$ and an ideal operator $\O\in\{\I_\subset,\I_\strsub,\Part_\lhd,\Ram,\Ram_\strsub\}$ we inductively define a sequence of ideals $\O^\alpha(I)$ on $P_\kappa A$ by letting
\begin{align*}
\O^0(I)&=I\\
\O^{\alpha+1}(I)&=\O(\O^\alpha(I))\\
\O^\alpha(I)&=\bigcup_{\beta<\alpha}\O^\beta(I)\text{ when $\alpha$ is a limit.}
\end{align*}

Ideals of the form $\I_\subset^\alpha(\NS_{\kappa,A})$, $\Part_\subset^\alpha(\NS_{\kappa,A})$, $\I_\strsub^\alpha(\NSS_{\kappa,A})$ and $\Part_\strsub^\alpha(\NSS_{\kappa,A})$ were studied by Kamo \cite{MR1419442}. In the remainder of the paper we prove several results involving the ideals $\O^\alpha(\Pi^1_\xi(\kappa,A))$ for $\xi<\kappa$, $\alpha<|A|^+$ and $\O\in\{\I_\strsub,\Ram_\strsub\}$. For example, recall that Baumgartner proved that the ineffable ideal on a cardinal is equal to the ideal generated by the subtle ideal and the $\Pi^1_2$-indescribability ideal. Generalizing this, we will show that in many cases these ideals $\O^\alpha(\Pi^1_\xi(\kappa,A))$, for $\O\in\{\I_\strsub,\Ram_\strsub\}$, can be obtained as the ideal generated by pair of smaller sub-ideals. We will also prove several hierarchy results. For example, it is easy to see that $\beta<\alpha$ implies $\O^\beta(I)\subseteq\O^\alpha(I)$. We will show that when the ideals involved are nontrivial it follows that $\beta<\alpha<|A|^+$ implies $\O^\beta(\Pi^1_\xi(\kappa,A))\subsetneq\O^\alpha(\Pi^1_\xi(\kappa,A))$. We will also show that as $\alpha$ increases, the large cardinal notions associated to $\O^\alpha(\Pi^1_\xi(\kappa,A))$ increase in consistency strength.

\subsection{Generating ideals}

The following lemma is due to Abe \cite[Theorem D]{MR1635559} in the case $\beta=1$. Versions of this lemma in the one-cardinal case were first established by Baumgartner \cite[Lemma 7.1]{MR0384553} and later by the first author \cite[Lemma 2.20]{MR4206111} as well as the first author and Peter Holy \cite{MR4594301}.

\begin{lemma}\label{lemma_ineffability_implies_indescribability}

Suppose $S\subseteq P_\kappa A$, $0<\beta<\kappa$ and for every $(S,\strsub)$-list $\vec{S}=\langle S_x\mid x\in S\rangle$ there is a $B\in\bigcap_{\xi<\beta}\Pi^1_\xi(\kappa,A)^+$ such that $B$ is homogeneous for $\vec{S}$. Then $S$ is a $\Pi^1_{\beta+1}$-indescribable subset of $P_\kappa A$.
\end{lemma}

\begin{proof}
Since $\kappa$ is Mahlo, there is a bijection $b:V_\kappa(\kappa, A)\to P_\kappa A$. By \cite[Lemma 1.3(4)]{MR1635559}, the set
\[C_b=\{x\in P_\kappa A\st b[V_{\kappa_x}(\kappa_x,x)]= P_{\kappa_x}x\}\]
is a weak club in $P_\kappa A$.

We proceed by induction on $\beta$. The base case in which $\beta=1$ is handled by \cite[Theorem D]{MR1635559}. The successor case is similar to \cite[Theorem D]{MR1635559}; we provide details for the reader's convenience. 

Suppose $\beta=\eta+1$ is a successor ordinal. To show that $S$ is $\Pi^1_{\beta+1}$-indescribable in $P_\kappa A$ it suffices to show that $T=S\cap C_b$ is $\Pi^1_{\beta+1}$-indescribable in $P_\kappa A$. Since $\I_\strsub(\bigcup_{\xi<\beta}\Pi^1_\xi(\kappa,A))$ is a strongly normal ideal on $P_\kappa A$, our assumption that every $(S,\strsub)$-list has a homogeneous set in $P(S)\cap \bigcap_{\xi<\beta}\Pi^1_\xi(\kappa,A)^+$ implies that every $(T,\strsub)$-list has a homogeneous set in $P(T)\cap \bigcap_{\xi<\beta}\Pi^1_\xi(\kappa,A)^+$.

Suppose $R\subseteq V_\kappa(\kappa,A)$ and suppose $\varphi$ is a $\Pi^1_{\eta+2}$ sentence of the form $\forall X\exists Y\psi$ where $\psi$ is $\Pi^1_\eta$ such that
\begin{align}(V_\kappa(\kappa,A),\in,R)\models \forall X\exists Y\psi.\label{eqn_univ}\end{align}
For contradiction, assume that for each $x\in T$ we have
\begin{align}(V_{\kappa_x}(\kappa_x,x),\in,R\cap V_{\kappa_x}(\kappa_x,x))\models\exists X\forall Y\lnot\psi.\label{eqn_build_a_list}
\end{align}
For each $x\in T$ let $A_x\subseteq V_{\kappa_x}(\kappa_x,x)$ witness (\ref{eqn_build_a_list}). Then $\vec{T}=\<b[A_x]\st x\in T\>$ is a $(T,\strsub)$ list, and so by our assumption on $S$, there is a $B\in P(T)\cap \bigcap_{\xi<\beta}\Pi^1_\xi(\kappa,A)^+$ homogeneous for $\vec{T}$. Let $X^*=\bigcup_{x\in B}A_x$, then by (\ref{eqn_univ}), there is a $Y^*\subseteq V_\kappa(\kappa,A)$ such that 
\[(V_\kappa(\kappa,A),\in,R,X^*,Y^*)\models\psi.\]
Since $B$ is $\Pi^1_\eta$-indescribable and $\psi$ is a $\Pi^1_\eta$ sentence, there is some $x\in B$ such that
\[(V_{\kappa_x}(\kappa_x,x),\in,R\cap V_{\kappa_x}(\kappa_x,x), X^*\cap V_{\kappa_x}(\kappa_x,x), Y^*\cap V_{\kappa_x}(\kappa_x,x))\models \psi.\]
By the homogeneity of $B$ we have $X^*\cap V_{\kappa_x}(\kappa_x,x)=A_x$, which contradicts \ref{eqn_build_a_list}.

Next suppose $\beta$ is a limit ordinal. As before, it suffices to show that $T=S\cap C_b$ is $\Pi^1_{\beta+1}$-indescribable. To this end, let $R\subseteq V_\kappa(\kappa,A)$ and let $\forall X\bigvee_{\xi<\beta}\varphi_\xi$ be a $\Pi^1_{\beta+1}$ sentence with
\[(V_\kappa(\kappa,A),\in,R)\models \forall X\bigvee_{\xi<\beta}\varphi_\xi.\]
For contradiction, suppose that for each $x\in T$ there is an $A_x\subseteq V_{\kappa_x}(\kappa_x,x)$ such that 
\begin{align}(V_{\kappa_x}(\kappa_x,x),\in,R\cap V_{\kappa_x}(\kappa_x,x),A_x)\models\bigwedge_{\xi<\beta}\lnot\varphi_\xi.\label{eqn_will_contradict}\end{align}
Then $\vec{T}=\<b[A_x]\st x\in T\>$ is a $(T,\strsub)$-list, hence there is a $B\in P(T)\cap \bigcap_{\xi<\beta}\Pi^1_\xi(\kappa,A)^+$ homogeneous for $\vec{T}$. Let $X^*=\bigcup_{x\in B}A_x$. Then for some $\xi<\beta$ we have 
\[(V_\kappa(\kappa,A),\in,R,X^*)\models\varphi_\xi,\]
and thus there is an $x\in B$ such that
\[(V_{\kappa_x}(\kappa_x,x),\in,R\cap V_{\kappa_x}(\kappa_x,x),X^*\cap V_{\kappa_x}(\kappa_x,x))\models\varphi_\xi,\]
but this contradicts \ref{eqn_will_contradict} since $X^*\cap V_{\kappa_x}(\kappa_x,x)=A_x$.
\end{proof}

Next we will show that ideals of the form $\I_\strsub(\Pi^1_\xi(\kappa,A))$ can be obtained as ideals generated by a pair of sub ideals, and furthermore, this leads to a characterization of the nontriviality of these ideals. For this result we will need the following two-cardinal notion of subtlety studied by Abe \cite[Definition 2.3]{MR2210149}. 

\begin{definition}
Suppose $\kappa$ is regular and $\kappa\leq|A|$. A set $S\subseteq P_\kappa A$ is \emph{strongly subtle} if for every $(S,\strsub)$-list $\vec{S}=\<S_x\subseteq P_{\kappa_x}x\st x\in S\>$ and every $C\in\NSS_{\kappa,A}^*$ there exists $y,z\in S\cap C$ with $y\strsub z$ and $S_y=S_z\cap P_{\kappa_y}y$. We let 
\[\NSSub_{\kappa,A}=\{X\subseteq P_\kappa A\st \text{$X$ is not strongly subtle}\}.\]
\end{definition}

Among other things, Abe proved \cite[Proposition 2.5(1)]{MR2210149} that $\NSSub_{\kappa,A}$ is a strongly normal ideal on $P_\kappa A$.

\begin{theorem}\label{theorem_strong_ineffability_ideal_characterization}
For all $n<\omega$, we have 
\[\I_\strsub(\Pi^1_n(\kappa,A))=\overline{\NSSub_{\kappa,A}\cup \Pi^1_{n+2}(\kappa,A)}.\footnote{Given a collection $A\subseteq P(X)$, where $X$ is some set, we write $\overline{A}$ to denote the ideal on $X$ generated by $A$. The set $X$ will be clear from the context.}\]
Furthermore, it is not the case the $P_\kappa A\notin \I_\strsub(\Pi^1_n(\kappa,A))$ is equivalent to $P_\kappa A\notin\NSSub_{\kappa,A}$ and $P_\kappa A\notin \Pi^1_{n+2}(\kappa,A)$, because if $\kappa$ is the least cardinal such that there is an $A$ with $\kappa\subseteq A$, $\kappa\leq|A|$ and $P_\kappa A\notin \I_\strsub(\Pi^1_n(\kappa,A))$ then there is an $x\in P_\kappa A$ such that $P_{\kappa_x}x$ is strongly subtle and $\Pi^1_{n+2}$-indescribable and yet $P_{\kappa_x}x\in\I_\strsub(\Pi^1_n(\kappa_x,x))$. 
\end{theorem}

\begin{proof}
Let $I = \overline{\NSSub_{\kappa, A}\cup \Pi^1_{n+2}(\kappa, A)}$. We show that $S \in \I_\strsub(\Pi^1_n(\kappa, A))^+$ if and only if $S \in I^+$

Suppose $S \in \I_\strsub(\Pi^1_n(\kappa, A))^+$. To show $S \in I^+$, it suffices to show $S$ is strongly subtle and $\Pi^1_{n+2}$-indescribable. Clearly $S$ is strongly subtle, and by Lemma \ref{lemma_ineffability_implies_indescribability} we know $S$ is $\Pi^1_{n+2}$-indescribable.  Thus $S \in I^+$. Conversely, suppose $S \in I^+$. For the sake of contradiction, suppose $S \in \I_\strsub(\Pi^1_n(\kappa,A))$. Then there is an $(S,\strsub)$-list $\vec{S}= \langle S_x \mid x \in S\rangle $ such that every homogeneous set for $\vec{S}$ is in the ideal $\Pi^1_n(\kappa, A)$. This is expressible by a $\Pi^1_{n+2}$-sentence $\varphi$ over $(V_{\kappa}(\kappa, A), \in  , \vec{S})$. Thus it follows that the set
\begin{align*}
    C &= \{ x \in P_\kappa A \mid (V_{\kappa_x}(\kappa_x, x), \in  , \vec{S}\cap V_{\kappa_x}(\kappa_x,x)) \models \varphi\} \\
    &= \{ x \in P_\kappa A \mid \text{every hom. set for } \vec{S}\restriction (P_{\kappa_x}x\cap S) \text{ is in } \Pi^1_n(\kappa_x, x)\}
\end{align*}
is in the filter $\Pi^1_{n+2}(\kappa, A)^*$.  Since $S \in I^+$, it follows that $S$ is not equal to the union of a non-strongly subtle set and a non-$\Pi^1_{n+2}$-indescribable set. Since $S = (S \cap C) \cup (S \backslash C)$ and $S \backslash C \in \Pi^1_{n+2}(\kappa,A)$, it follows that $S \cap C$ must be strongly subtle. As a direct consequence of \cite[Theorem B]{MR2210149}, there is some $x \in S \cap C$ for which there is an $H \subseteq S \cap C \cap P_{\kappa_x} x$ which is $\Pi^1_n$-indescribable in $P_{\kappa_x}x$ and homogeneous for $\vec{S}$. This contradicts $x \in C$.

For the remaining statement, let $\kappa$ be the least cardinal such that there is an $A$ with $P_\kappa A\notin\I_\strsub(\Pi^1_n(\kappa,A))$. We show there are many $x \in P_\kappa A$ for which $P_{\kappa_x}x$ is both subtle and $\Pi^1_{n+2}$-indescribable.  The fact that $P_\kappa A$ is strongly subtle can be expressed by a $\Pi^1_1$-sentence $\varphi$ over $V_\kappa(\kappa,A)$ and thus the set
\[ C = \{x \in P_\kappa A \mid (V_{\kappa_x}(\kappa_x,x), \in ) \models \varphi \} = \{ x \in P_\kappa A \mid P_{\kappa_x}x \text{ is strongly subtle} \} \]
is in the filter $\Pi^1_1(\kappa,A)^* \subseteq \Pi^1_{n+2}(\kappa,A)^* $.  Furthermore, by \cite[Lemma 3.8]{MR2210149} the set 
\[H = \{ x \in P_\kappa A \mid P_{\kappa_x}x \text{ is } \Pi^1_{n+2}\text{-indescribable} \} \]
is in the filter $\NSSub_{\kappa,A}^*$. Since $\I_\strsub (\Pi^1_n(\kappa, A)) \supseteq \NSSub_{\kappa,A}\cup \Pi^1_{n+2}(\kappa,A)$, it follows that $C \cap H$ is in the filter $\I_\strsub(\Pi^1_n(\kappa,A))^*$  
\end{proof}

\begin{remark}
The previous theorem can be generalized to ideals of the form $\I_\strsub^\alpha(\Pi^1_\xi(\kappa,A))$, where $\alpha,\xi<\kappa$ as is done in \cite{MR4206111} and \cite{MR4594301}. For example, to obtain $\I_\strsub^2(\Pi^1_\xi(\kappa,A))$ as the ideal generated by two proper sub-ideals, one must replace the strongly subtle ideal $\NSSub_{\kappa,A}$ with an ideal defined using a \emph{pre-operator}. The details are left to the interested reader.
\end{remark}


Next, in order to prove a version of Theorem \ref{theorem_strong_ineffability_ideal_characterization} for $\Ram_\strsub(\Pi^1_n(\kappa,A))$, we introduce another new large cardinal notion and an assoicated ideal. The following definition can be viewed as a generalization of the notion of pre-Ramseyness introduced in \cite{MR0540770} and later studied in \cite{CodyLCI, MR4206111, MR4594301, MR1077260}.

\begin{definition}
Suppose $\kappa$ is regular and $\kappa\leq |A|$. Further suppose that $\vec{I}=\<I_x\st x\in P_\kappa A\>$ is a function such that for each $x\in P_\kappa A$, $I_x$ is an ideal on $P_{\kappa_x}x$. We define an ideal $\Ram_{\strsub}^\pre(\vec{I})$ on $P_\kappa A$ by letting $S\in\Ram_\strsub^\pre(\vec{I})^+$ if and only if for every $\strsub$-regressive function $f:[S]^{<\omega}_\strsub\to P_\kappa A$ and every $C\in \NSS_{\kappa,A}^*$ there is some $x\in S\cap C$ such that there is an $H\in P(S\cap C\cap P_{\kappa_x}x)\cap I_x^+$ homogeneous for $f$.
\end{definition}

In the case where the ideals listed by the function $\vec{I}$ have a uniform definition, we will often use the notation $\Ram_\strsub^\pre(I)=\Ram_\strsub^\pre(\vec{I})$, where $I_{\kappa,A}$ is the relevant ideal on $P_\kappa A$. For example, if $\vec{I}=\<\NSS_{\kappa_x,x}\st x\in P_\kappa A\>$, when we write $\Ram_\strsub^\pre(\NSS_{\kappa,A})$ we mean $\Ram_\strsub^\pre(\vec{I})$.

\begin{theorem}\label{theorem_strong_ramsey_ideal_characterization}
For all $n<\omega$, 
\[\Ram_\strsub(\Pi^1_n(\kappa,A))=\overline{\Ram_\strsub^\pre(\Pi^1_n(\kappa,A))\cup \Pi^1_{n+2}(\kappa,A)}.\]
Furthermore, it is not the case that $P_\kappa A\notin \Ram_\strsub(\Pi^1_n(\kappa,A))$ is equivalent to $P_\kappa A\notin \Ram_\strsub^\pre(\Pi^1_n(\kappa,A))$ and $P_\kappa A\notin \Pi^1_{n+2}(\kappa,A)$, because if $\kappa$ is the least cardinal such that there is an $A$ with $\kappa\subseteq A$, $\kappa\leq|A|$ and $P_\kappa A\notin \Ram_\strsub(\Pi^1_n(\kappa,A))$ then there is an $x\in P_\kappa A$ such that the ideals $\Ram_\strsub^\pre(\Pi^1_n(\kappa_x,x))$ and $\Pi^1_n(\kappa_x,x)$ are nontrivial and yet $P_{\kappa_x}x\in\Ram_\strsub(\Pi^1_n(\kappa_x,x))$. 
\end{theorem}

The proof of Theorem \ref{theorem_strong_ramsey_ideal_characterization} is similar to that of Theorem \ref{theorem_strong_ineffability_ideal_characterization}, the only difference being that one must work with regressive functions or $(\omega,S,\strsub)$-lists instead of $(S,\strsub)$-lists. For similar results in the one-cardinal context see \cite{MR4206111, MR4594301}.

\subsection{Hierarchy results} In this section we prove several hierarchy results concerning ideals of the form $\I^\alpha(\Pi^1_\xi(\kappa,A))$ and $\Ram_\strsub^\alpha(\Pi^1_\xi(\kappa,A))$, where $\kappa$ is regular, $A$ is a set of ordinals with $\kappa\leq |A|$, $\xi<\kappa$ and $\alpha<|A|^+$. In order to handle cases in which $\alpha>\kappa$, let us briefly outline some important properties of canonical functions that we will require (see \cite[Section 2.6]{MR2768692} and \cite[Section 2]{MR725730}).

Given ordinal valued functions $f$ and $g$ with domain $P_\kappa A$ we write $f\sim g$ if and only if $\{x\in P_\kappa A\st f(x)=g(x)\}$ contains a club, and similarly for $f\leq g$ and $f<g$. It is easy to see that $\sim$ is an equivalence relation, $\leq$ is transitive and reflexive and that $<$ is transitive and well-founded. For each $f$ we let $\|f\|$ be the rank of $f$ with respect to $<$. We say that such a function $f$ is \emph{canonical} if and only if for every $g$, $\|f\|\leq\|g\|$ implies $f\leq g$; in other words, $f$ is canonical if it is minimal in the $\leq$ ordering among all ordinal-valued functions on $P_\kappa A$ of the same rank. Notice that when $f$ is canonical, $\|f\|<\|g\|$ easily implies that $f<g$.

\begin{lemma}\label{lemma_canonical} Suppose $\kappa$ is a regular uncountable cardinal and $A$ is a set of ordinals with $\kappa\leq |A|$. There is a sequence $\<f_\alpha\st\alpha<|A|^+\>$ of ordinal-valued functions defined on $P_\kappa A$ such that for all $\alpha<|A|^+$ it follows that
\begin{enumerate}
\item $f_\alpha$ is a canonical function with rank $\alpha$,
\item whenever $x\in P_\kappa A$ is such that $x\cap\kappa$ is regular and uncountable we have $f_\alpha\restrict P_{x\cap\kappa}x$ is canonical on $P_{x\cap\kappa}x$ of rank $f_\alpha(x)$ and
\item the set $\{x\in P_\kappa A\st f_\alpha(x)<|x|^+\}$ is club in $P_\kappa A$.
\end{enumerate}
\end{lemma}

The proof of Lemma \ref{lemma_canonical} is standard and is left to the reader. For example, Baldwin established the existence of a sequence $\<f_\alpha\st\alpha<|A|^+\>$ satisfying \ref{lemma_canonical}(1) and \ref{lemma_canonical}(2) for all $\alpha$ (see \cite[Theorem 2.12]{MR725730}), and the fact that (3) can be obtained for all $\alpha$ is implicit in Baldwin's proof. Let us also remark, that one can also prove Lemma \ref{lemma_canonical} by using the definition of $\<f_\alpha\st\alpha<|A|^+\>$ stated by Foreman \cite[Section 2.6]{MR2768692} and the fact that each $f_\alpha$ provides a representative of the ordinal $\alpha$ in any generic ultrapower obtained by forcing with $P(P_\kappa A)/I$ where $I$ is a countably complete normal ideal on $P_\kappa A$ (see \cite[Proposition 2.34]{MR2768692}).

\begin{lemma}\label{lemma_pos_union_of_pos}
Suppose $\kappa$ is a regular uncountable cardinal, $A$ is a set of ordinals with $\kappa\leq |A|$, $\xi<\kappa$, $\alpha<|A|^+$ and $\O\in\{\I_\strsub,\Ram_\strsub\}$. If $S\in \O^\alpha(\Pi^1_\xi(\kappa,A))^+$ and for each $x\in S$ we have a set $S_x\in O^{f_\alpha(x)}(\Pi^1_\xi(\kappa_x,x))^+$, then $\bigcup_{x\in S}S_x\in\O^\alpha(\Pi^1_\xi(\kappa,A))^+$.
\end{lemma}

\begin{proof}
We provide a proof for the case in which $\O=\R$; the case in which $\O=\I$ is essentially the same, only one must replace regressive functions by lists.

Suppose $\alpha=0$. Suppose $S\in\Pi^1_\xi(\kappa,A)^+$ and for each $x\in S$ we have $S_x\in \Pi^1_\xi(\kappa_x,x)^+$. We must show that $\bigcup_{x\in S}S_x\in\Pi^1_\xi(\kappa,A)^+$. Fix $R\subseteq V_\kappa(\kappa,A)$ and let $\varphi$ be a $\Pi^1_\xi$ sentence such that $(V_\kappa(\kappa,A),\in,R)\models\varphi$. Since $S\in\Pi^1_\xi(\kappa,A)^+$, there is an $x\in S$ such that $(V_{\kappa_x}(\kappa_x,x),\in,R\cap V_{\kappa_x}(\kappa_x,x))\models\varphi$. Now since $S_x\in \Pi^1_\xi(\kappa_x,x)^+$, there is a $y\in S_x$ such that $(V_{\kappa_y}(\kappa_y,y),\in,R\cap V_{\kappa_y}(\kappa_y,y))\models\varphi$. Hence $\bigcup_{x\in S}S_x\in\Pi^1_\xi(\kappa,A)^+$.

Now, suppose $\alpha=\eta+1>0$ is a successor ordinal and the result holds for $\eta$. Fix a $\strsub$-regressive function $f:[\bigcup_{x\in S}S_x]_\strsub^{<\omega}\to P_\kappa A$. Fix a club $C_0\subseteq P_\kappa A$ such that $x\in C_0$ implies $f_\alpha(x)=f_{\eta}(x)+1$. By assumption, for each $x\in S\cap C_0$ we have $S_x\in \Ram_\strsub^{f_\eta(x)+1}(\Pi^1_\xi(\kappa_x,x))^+$, and thus there is a set $H_x\in P(S_x)\cap \Ram_\strsub^{f_\eta(x)}(\Pi^1_\xi(\kappa_x,x))^+$ homogeneous for $f\restrict[S_x]_\strsub^{<\omega}\to P_\kappa A$. Since $S\in \Ram_\strsub^\alpha(\Pi^1_\xi(\kappa,A))^+$, it easily follows that the $(S,\strsub)$-sequence $\vec{H}=\<H_x\st x\in S\>$ has a homogeneous set $H\in P(S)\cap \Ram_\strsub^\eta(\Pi^1_\xi(\kappa,A))^+$ (just extend the $(S,\strsub)$-sequence to any $(\omega,S,\strsub)$-sequence). By our inductive hypothesis, $\bigcup_{x\in H}H_x\in \Ram_\strsub^\eta(\Pi^1_\xi(\kappa,A))^+$. Now it is easy to verify that $\bigcup_{x\in H}H_x$ is homogeneous for $f$.

If $\alpha$ is a limit ordinal and the result holds for ordinals less than $\alpha$, it is easy to verify that the result holds for $\alpha$ using the fact that $\Ram_\strsub^\alpha(\Pi^1_\xi(\kappa,A))=\bigcup_{\eta<\alpha}\Ram_\strsub^\eta(\Pi^1_\xi(\kappa,A))$.

\end{proof}

To prove a hierarchy result (Theorem \ref{theorem_hierarchy_1}), we need the following.

\begin{lemma}[{\cite[Theorem 2.12]{MR725730}}]\label{lemma_baldwin_2} Suppose $\kappa$ is a regular uncountable cardinal and $A$ is a set of ordinals with $\kappa\leq|A|$. The following properties of canonical functions on $P_\kappa A$ hold.
\begin{enumerate}
\item[(a)] If $f\leq g$ and $g\leq f$ then $\{x\in P_\kappa A\st f(x)=g(x)\}$ is in the club filter on $P_\kappa A$.
\item[(b)] If $f$ and $g$ are both canonical on $P_\kappa A$ and $\|f\|=\|g\|$ then $f$ and $g$ are equal on a club.
\item[(c)] If $f$ is canonical on $P_\kappa A$ and $g(x)=f(x)+1$ for club-many $x\in P_\kappa A$, then $g$ is canonical and $\|g\|=\|f\|+1$.
\item[(d)] If $\<f_\gamma\st\gamma\in A\>$ is a seqeunce of canonical functions on $P_\kappa A$ and $f$ is an ordinal-valued function on $P_\kappa A$ defined by $f(x)=\bigcup_{\eta\in x}f_\eta(x)$, then $f$ is canonical and $\|f\|=\bigcup_{\eta\in A}\|f_\eta\|$.
\end{enumerate}

\end{lemma}

\begin{lemma}
Suppose $\kappa$ is a regular uncountable cardinal, $A$ is a set of ordinals with $\kappa\leq |A|$, $\xi<\kappa$, $\alpha<|A|^+$ and $\O\in\{\I_\strsub,\Ram_\strsub\}$. If $P_\kappa A\in \O^\alpha(\Pi^1_\xi(\kappa,A))^+$ where $\alpha<|A|^+$, then the set
\[X_\alpha=\{x\in P_\kappa A\st P_{\kappa_x}x\in \O^{f_\alpha(x)}(\Pi^1_\xi(\kappa_x,x))\}\]
is in $\O^\alpha(\Pi^1_\xi(\kappa,A))^+$.
\end{lemma}

\begin{proof}
We will prove this for $\O=\Ram_\strsub$; the case in which $\O=\I_\strsub$ is similar.

We follow \cite[Theorem 5.2]{MR1077260} and proceed by induction on $\kappa$. We assume the result holds for all cardinals less than $\kappa$ and prove that it holds for $\kappa$. If 
\[S=\{x\in P_\kappa A\st P_{\kappa_x}x\in\Ram_\strsub^{f_\alpha(x)}(\Pi^1_\xi(\kappa_x,x))^+\}\] 
is in $\Ram_\strsub^\alpha(\Pi^1_\xi(\kappa,A))$, then $X_\alpha\in \Ram_\strsub^\alpha(\Pi^1_\xi(\kappa,A))^*$ and we are done. So, we assume that $S\in\Ram_\strsub^\alpha(\Pi^1_\xi(\kappa,A))^+$.

For each $z\in P_\kappa A$, we let $\<f^z_\eta\st\eta<|z|^+\>$ denote a sequence of canonical functions defined on $P_{\kappa_z} z$ satisfying conditions analogous to Lemma \ref{lemma_canonical}(1)-(3). Let $C_\alpha\subseteq P_\kappa A$ be a club such that for all $z\in C_\alpha$ the following properties hold:
\begin{enumerate}
\item $z\cap\kappa<\kappa$,
\item $f_\alpha(z)<|z|^+$,
\item when $z\cap\kappa$ is a regular uncountable cardinal we have that $f_\alpha\restrict P_{\kappa_z}z$ is a canonical function on $P_{\kappa_z}z$ of rank $f_\alpha(z)$, and thus $f_\alpha\restrict P_{\kappa_z}z=f^z_{f_\alpha(z)}$.
\end{enumerate}
Let us show that for each $z\in C_\alpha\cap S$ with $\kappa_z>\xi$, the set $X_\alpha\cap P_{\kappa_z}z$ is in $\Ram_\strsub^{f_\alpha(z)}(\Pi^1_\xi(\kappa_z,z))^+$; then, by Lemma \ref{lemma_pos_union_of_pos}, it will follow that $X_\alpha\in\Ram_\strsub^\alpha(\Pi^1_\xi(\kappa,A))^+$. Fix $z\in C_\alpha\cap S$. Notice that $\kappa_z<\kappa$ and $f_\alpha(z)<|z|^+$. Thus, by our inductive hypothesis, the set
\[\{x\in P_{\kappa_z}z\st P_{\kappa_x}x\in \Ram_\strsub^{f_{f_\alpha(z)}^{z}(x)}(\Pi^1_\xi(\kappa_x,x))\}\]
is in $\Ram_\strsub^{f_\alpha(z)}(\Pi^1_\xi(\kappa_z,z))^+$. But, since $z\in C_\alpha$ we have $f^z_{f_\alpha(z)}(x)=f_\alpha(x)$, and thus the set 
\[X_\alpha\cap P_{\kappa_z}z=\{x\in P_{\kappa_z}z\st P_{\kappa_x}x\in \Ram_\strsub^{f_\alpha(x)}(\Pi^1_\xi(\kappa_x,x))\}\]
is in $\Ram_\strsub^{f_\alpha(z)}(\Pi^1_\xi(\kappa_z,z))^+$.

\end{proof}

\begin{theorem}\label{theorem_hierarchy_1}
Suppose $\kappa$ is a regular uncountable cardinal, $A$ is a set of ordinals with $\kappa\leq |A|$, $\xi<\kappa$, $\alpha<|A|^+$ and $\O\in\{\I_\strsub,\Ram_\strsub\}$. If $P_\kappa A\in\Ram_\strsub^{\alpha+1}(\Pi^1_\xi(\kappa,A))^+$, then for all $\beta\leq\alpha$ and all sets $X\in \Ram_\strsub^\beta(\Pi^1_\xi(\kappa,A))^+$, it follows that the set
\[\{x\in P_\kappa A\st X\cap P_{\kappa_x}x\in\Ram_\strsub^{f_\beta(x)}(\Pi^1_\xi(\kappa_x,x))^+\}\]
is in the filter $\Ram_\strsub^{\beta+1}(\Pi^1_\xi(\kappa,A))^*$.
\end{theorem}

\begin{proof}
We give a proof for the case $\O=\Ram_\strsub$. The proof for $\I_\strsub$ is similar, but uses lists instead of $\strsub$-regressive functions.

Following \cite[Theorem 5.3]{MR1077260}, we proceed by induction on $\beta$.

Suppose $\beta=0$. The assumption that $P_\kappa A\in \Ram_\strsub^{\alpha+1}(\Pi^1_\xi(\kappa,A))^+$ implies that $P_\kappa A$ is $\Pi^1_{\xi+1}$-indescribable, and since the fact that $X\in\Ram_\strsub^0(\Pi^1_\xi(\kappa,A))^+=\Pi^1_\xi(\kappa,A))^+$ is expressible by a $\Pi^1_{\xi+1}$ sentence over $(V_\kappa(\kappa,A),\in,X)$, it follows the set
\[\{x\in P_\kappa A\st X\cap P_{\kappa_x}x\in\Pi^1_\xi(\kappa_x,x)^+\}\]
is in the filter $\Pi^1_{\xi+1}(\kappa,A)^*\subseteq\Ram_\strsub(\Pi^1_\xi(\kappa,A))^*$ (the last containment follows from Lemma \ref{lemma_ineffability_implies_indescribability}).

Suppose $\beta=\eta+1$. Let $C_0\subseteq P_\kappa A$ be a club such that $x\in C_0$ implies $f_\beta(x)=f_\eta(x)+1$. Suppose $X\in\Ram_\strsub^\beta(\Pi^1_\xi(\kappa,A))$. By our inductive hypothesis the set
\[\{x\in C_0\st X\cap P_{\kappa_x}x\in\Ram_\strsub^{f_\eta(x)}(\Pi^1_\xi(\kappa_x,x))^+\}\]
is in the filter $\Ram_\strsub^{\eta+1}(\Pi^1_\xi(\kappa,A))^*$ and is hence also in the filter $\Ram_\strsub^{\beta+1}(\Pi^1_\xi(\kappa,A))^*$. 

Now let
\[T=\{x\in P_\kappa A\st X\cap P_{\kappa_x}x\in\Ram_\strsub^{f_\eta(x)+1}(\Pi^1_\xi(\kappa_x,x))\}.\]
It will suffice to show that $T\in\Ram_\strsub^{\beta+1}(\Pi^1_\xi(\kappa,A))$. For a contradiction, suppose $T\in\Ram_\strsub^{\beta+1}(\Pi^1_\xi(\kappa,A))^+$. Then the set
\[Y=\{x\in C_0\st X\cap P_{\kappa_x}x\in \Ram_\strsub^{f_\eta(x)}(\Pi^1_\xi(\kappa_x,x))^+\cap\Ram_\strsub^{f_\eta(x)+1}(\Pi^1_\xi(\kappa_x,x))\}\]
is in $\Ram_\strsub^{\beta+1}(\Pi^1_\xi(\kappa,A))^+$.

For each $x\in Y$, let $g_x:[X\cap P_{\kappa_x}x]_\strsub^<\to P_{\kappa_x}x$ be a $\strsub$-regressive function with no homogeneous set in $\Ram_\strsub^{f_\eta(x)}(\Pi^1_\xi(\kappa,A))^+$.

Fix a bijection $b:(P_\kappa A)\times(P_\kappa A)\to P_\kappa A$ and note that the set
\[C_1=\{x\in P_\kappa A\st b[P_{\kappa_x}x\times P_{\kappa_x}x]=P_{\kappa_x}x\}\]
is a weak club in $P_\kappa A$. Now let $Z=Y\cap C_1$. For each $x\in Z$ let $Z_x=b[g_x]\subseteq P_{\kappa_x}x$. This defines a $Z$-list $\vec{Z}=\<Z_x\st x\in Z\>$. By assumption, there is a set $B\in \Ram_\strsub^\beta(\Pi^1_\xi(\kappa,A))^+$ homogeneous for $\vec{Z}$. Let $f=\bigcup\{g_x\st x\in B\}$. Then $f$ is $\strsub$-regressive on $[X]_\strsub^{<}$.

Since $X\in\Ram_\strsub^{\eta+1}(\Pi^1_\xi(\kappa,A))^+$, there is an $H\in P(X)\cap \Ram_\strsub^\eta(\Pi^1_\xi(\kappa,A))^+$ homogeneous for $f$. By induction, the set
\[\{x\in P_\kappa A\st H\cap P_{\kappa_x}x\in \Ram_\strsub^{f_\eta(x)}(\Pi^1_\xi(\kappa,A))^+\}\]
is in $\Ram_\strsub^\beta(\Pi^1_\xi(\kappa,A))^*$. Choose $x\in B$ such that $H\cap P_{\kappa_x}x\in\Ram_\strsub^{f_\eta(x)}(\Pi^1_\xi(\kappa_x,x))^+$. Then $f\restrict[X\cap P_{\kappa_x}x]_\strsub^{<\omega}=g_x$ and $H\cap P_{\kappa_x}x$ is homogeneous for $g_x$, a contradiction.

Now suppose $\beta\leq\alpha$ is a limit ordinal and $X\subseteq P_\kappa A$ is in $\Ram_\strsub^\beta(\Pi^1_\xi(\kappa,A))^+$. If $\eta<\beta$ then $X\in\Ram_\strsub^\eta(\Pi^1_\xi(\kappa,A))^+$ since $\Ram_\strsub^\beta(\Pi^1_\xi(\kappa,A))=\bigcup_{\eta<\beta}\Ram_\strsub^\eta(\Pi^1_\xi(\kappa,A))$. Thus, by our inductive hypothesis, for each $\eta<\beta$ the set 
\[D_\eta=\{x\in P_\kappa A\st X\cap P_{\kappa_x}x\in \Ram_\strsub^{f_\eta(x)}(\Pi^1_\xi(\kappa_x,x))^+\}\]
is in the filter $\Ram_{\strsub}^{\eta+1}(\Pi^1_\xi(\kappa,A))^*$. Thus, each $D_\eta$ is in the filter $\Ram_\strsub^\beta(\Pi^1_\xi(\kappa,A))^*$ and thus also in $\Ram_\strsub^{\beta+1}(\Pi^1_\xi(\kappa,A))^*$, which is nontrivial and strongly normal. By normality, the set 
\[\bigtriangleup_{\eta<\beta}D_\eta=\{x\in P_\kappa A\st x\in\bigcap_{\eta\in x}D_\eta\}\]
is in the filter $\Ram_\strsub^{\beta+1}(\Pi^1_\xi(\kappa,A))^*$. Applying Lemma \ref{lemma_baldwin_2}(d) to the sequence $\<f_\eta\st\eta<\beta\>$ (using a reindexing if necessary), it follows that the function $x\mapsto \bigcup_{\eta\in x}f_\eta(x)$ is canonical on $P_\kappa A$ of rank $\bigcup_{\eta<\beta}\|f_\eta\|=\beta$. Therefore, the set 
\[C=\{x\in P_\kappa A\st \bigcup_{\eta\in x}f_\eta(x)=f_\beta(x)\}\]
is club in $P_\kappa A$. Hence the set $C\cap\bigtriangleup_{\eta<\beta}D_\eta$, which is contained in 
\[\{x\in P_\kappa A\st X\cap P_{\kappa_x}x\in\Ram_\strsub^{f_\beta(x)}(\Pi^1_\xi(\kappa_x,x))^+\},\] is in the filter $\Ram_\strsub^{\beta+1}(\Pi^1_\xi(\kappa,A))^*$.
\end{proof}

\begin{corollary}
Suppose $\kappa$ is a regular uncountable cardinal, $A$ is a set of ordinals with $\kappa\leq |A|$, $\xi<\kappa$, $\alpha<|A|^+$ and $\O\in\{\I_\strsub,\Ram_\strsub\}$. If the ideal $\O^\alpha(\Pi^1_\xi(\kappa,A))$ is nontrivial then 
\[\O^\alpha(\Pi^1_\xi(\kappa,A))\subsetneq\O^{\alpha+1}(\Pi^1_\xi(\kappa,A)).\]
\end{corollary}

Next we will generalize a theorem of Baumgartner \cite{MR0384553} and results of the first author and Peter Holy \cite{MR4594301}, by proving a theorem which establishes, among other things, that the existence of cardinals $\kappa\leq\lambda$ such that $\I^2_\strsub(I_{\kappa,\lambda})$ is nontrivial is strictly stronger in consistency strength than the existence of cardinals $\kappa\leq\lambda$ for which $\I_\strsub(\Pi^1_\beta(\kappa,\lambda))$ is nontrivial for all $\beta<\kappa$. This theorem strengthens Theorem \ref{theorem_hierarchy_1} in the case where $\O=\I_\strsub$. See the comments after the proof of Theorem \ref{theorem_strong_ineffable_hierarchy} for more information on generalizing Theorem \ref{theorem_strong_ineffability_ideal_characterization} to $\Ram_\strsub$.

\begin{theorem}\label{theorem_strong_ineffable_hierarchy}
Suppose $\kappa$ is a regular uncountable cardinal, $A$ is a set of ordinals with $\kappa\leq |A|$, $\alpha<|A|^+$ and $S\in \I_\strsub^{\alpha+1}(I_{\kappa,A})^+$. Suppose $\vec{S}=\<S_x\st x\in S\>$ is an $(S,\strsub)$-list. Let
\begin{align*}
Z = \{x \in S \mid (\exists X \subseteq S \cap P_{\kappa_x}x) &(\forall \beta < \kappa_x \ X \in \I_\strsub^{f_{\alpha}(x)}(\Pi^1_\beta(\kappa_x,x))^+) \wedge \\ 
	&(X \cup \{x\} \text{ is homog. for $\vec{S}$}) \}
\end{align*}
Then $S\setminus Z\in \I_\strsub^{\alpha+1}(I_{\kappa,A})$.
\end{theorem}

\begin{proof}
We proceed by induction on $\alpha<|A|^+$. The case in which $\alpha=0$ follows firectly from an argument given by Abe \cite[Lemma 3.8]{MR2210149}, which is a straightforward generalization of Baumgartner's \cite[Theorem 4.1]{MR0384553}; the arguments for the successor case and the limit case are similar. Let us provide a proof for the successor case. The interested reader may easily piece together a proof of the limit case by consulting the following successor case and the detailed arguments in \cite{MR4594301}.

Suppose $\alpha=\delta+1<|A|^+$ is a successor ordinal, and suppose for a contradiction that $S\setminus Z\in \I_\strsub^{\delta+2}(I_{\kappa,A})^+$. By Lemma \ref{lemma_baldwin_2} we may let $C$ be a club subset of $P_\kappa A$ such that $x\in C$ implies $f_{\delta+1}(x)=f_\delta(x)+1$. The set
\[E=\{x\in S\setminus Z\st \kappa_x\text{ is inaccessible}\}\cap C\]
is in $\I^{\delta+2}(I_{\kappa,A})^+$. For each $x\in E$, let $B_x=\{y\in S\cap P_{\kappa_x}x\st S_y=S_x\cap P_{\kappa_y}y\}$. Since $B_x\cup\{x\}$ is homogeneous for $\vec{S}$ and $x\in S\setminus Z$, there is an ordinal $\xi_x<\kappa_x$ such that $B_\alpha\in \I^{f_\delta(x)+1}(\Pi^1_{\xi_x}(\kappa_x,x))$, and hence we may fix a $(B_x,\strsub)$-list $\vec{B}^x=\<b^x_y\st y\in B_x\>$ such that $\vec{B}^x$ gas no homogeneous set in $\I^{f_\delta(x)}(\Pi^1_{\xi_x}(\kappa_x,x))^+$.

Since $E\in \I^{\delta+2}(I_{\kappa,A})^+$, there is an $H\in P(E)\cap \I^{\delta+1}(I_{\kappa,A})^+$ such that whenever $y\strsub x$ and $x,y\in H$ we have $S_y=S_x\cap P_{\kappa_y}y$, $B_y=B_x\cap P_{\kappa_y}y$ and $\vec{B}^y=\vec{B}^x\restrict B_y$. Let $D=\bigcup_{x\in H}S_x$, $B=\bigcup_{x\in H}B_x$ and $\vec{B}=\bigcup_{x\in H}\vec{B}^x=\<b_x\st x\in B\>$. Since $B=\{x\in P_\kappa A\st S_x=D\cap P_{\kappa_x}x\}$, it follows that $H\subseteq B$.

Now let $A_0$ be the set of all $x\in H$ such that there is an $X\subseteq P\cap P_{\kappa_x}x$ such that
\[(\forall \xi<\kappa_x\ X\in \I^{f_\delta(x)}(\Pi^1_{\xi}(\kappa_x,x))^+) \land (X\cup\{x\}\text{ is hom. for }\vec{B}).\]
By our inductive hypothesis, $H\setminus A_0\in \I^{\delta+1}(I_{\kappa,A})$, and hence $A_0\in \I^{\delta+1}(I_{\kappa,A})^+$. Thus, there is an $x\in A_0$. Since $x\in H$, it follows by homogeneity that $\vec{B}\restrict (H\cap P_{\kappa_x}x)=\vec{B}^x\restrict H$. But, by the definition of $A_0$, and since $\xi_x<\kappa_x$, there is some $X\in P(H\cap P_{\kappa_x}x)\cap \I^{f_\delta(x)}(\Pi^1_{\xi_x}(\kappa_x,x))^+$ which is homogeneous for $\vec{B}^x$, a contradiction.
\end{proof}

At the time of writing this article, the authors did not know whether Theorem \ref{theorem_strong_ineffable_hierarchy} holds if we replace $\I_\strsub$ with $\Ram_\strsub$ and $\vec{S}$ with an $(\omega,S,\strsub)$-list. In fact, at that time, it was not known whether the corresponding result holds for the single cardinal Ramsey operator. See \cite{MR4594301} for a detailed discussion about the problems involved with generalizing Theorem \ref{theorem_strong_ineffable_hierarchy} to the Ramsey operator in the one-cardinal case. Since the current article was written, the first author, Lambie-Hanson and Zhang proved theorems analogous to Theorem \ref{theorem_strong_ineffable_hierarchy} for both the single cardinal Ramsey operator and $\Ram_\strsub$ (see \cite[Section 4]{CLHZ2024}).


\end{document}